\documentclass[a4paper,11pt]{article}

\usepackage{array}
\usepackage{authblk}
\usepackage{amsmath,amsthm}
\usepackage{mathtools}
\usepackage{graphics}
\usepackage{epsfig}
\usepackage{latexsym}
\usepackage{amssymb}
\usepackage{color}
\usepackage[ruled,lined,algonl]{algorithm2e}
\usepackage{enumerate}
\usepackage{dsfont}

\usepackage[T1]{fontenc}
\usepackage{longtable}
\usepackage[margin=2.54cm]{geometry}

\usepackage{times}
\usepackage[compact]{titlesec}  
\titlespacing{\section}{10pt}{10pt}{10pt}

\newcommand{\Hom}{\mathbf{H}}
\newcommand{\Cycles}{\mathbf{Z}}
\newcommand{\Boundaries}{\mathbf{B}}
\newcommand{\Chains}{\mathbf{C}}
\newcommand{\C}{\mathbb{C}}
\newcommand{\co}{\colon\thinspace}

\newcommand{\tri}{\mathfrak{T}}
\newcommand{\M}{M}

\newcommand{\tv}{\mathrm{TV}}
\newcommand{\adm}{\mathrm{Adm}}

\newcommand{\Z}{\mathbb{Z}}
\newcommand{\zmin}{z^-}
\newcommand{\zmax}{z^+}

\newtheorem{thm}{Theorem}
\newtheorem{lemma}{Lemma}

\newtheorem{corollary}{Corollary}
\theoremstyle{definition}

\usepackage{xhfill}
\newcommand{\ditto}[1][.4pt]{\xrfill{#1}~\textquotedbl~\xrfill{#1}}

\title{A polynomial time algorithm to compute quantum invariants of 
$3$-manifolds with bounded first Betti number.\thanks{This work is supported by
the Australian Research Council (projects DP140104246 and DP150104108).}}

\author{Cl\'ement Maria}
\author{Jonathan Spreer}
\affil{The University of Queensland, Australia\\
\texttt{c.maria@uq.edu.au}, \
  \texttt{j.spreer@uq.edu.au}}
   
\begin{document}

\maketitle

\begin{abstract}
In this article, we introduce a fixed parameter tractable algorithm for computing the Turaev-Viro invariants $\tv_{4,q}$, using the dimension of the first homology group of the manifold as parameter.

This is, to our knowledge, the first parameterised algorithm in computational $3$-manifold topology using a topological parameter. The computation of $\tv_{4,q}$ is known to be \#P-hard in general; using a topological parameter provides an algorithm polynomial in the size of the input triangulation for the extremely large family of $3$-manifolds with first homology group of bounded rank.

Our algorithm is easy to implement and running times are comparable with running times to compute integral homology groups for standard libraries of triangulated $3$-manifolds. The invariants we can compute this way are powerful: in combination with integral homology and using standard data sets we are able to roughly double the pairs of $3$-manifolds we can distinguish.

We hope this qualifies $\tv_{4,q}$ to be added to the short list of standard properties (such as orientability, connectedness, Betti numbers, etc.) that can be computed ad-hoc when first investigating an unknown triangulation.
\end{abstract}


\newpage

\section{Introduction}
\label{sec:intro}

In geometric topology, {\em invariants} are properties of manifolds telling
pairs of non-homeomorphic manifolds apart. Invariants are important, 
since deciding whether two manifolds are topologically equivalent---the 
so-called {\em homeomorphism
problem}---is remarkably difficult in dimension three, and undecidable in 
dimensions four and higher. Thus, invariants help to settle this
potentially unsolvable, yet fundamentally important problem in many, albeit not all cases.

This article is concerned with the special case of $3$-dimensional manifolds,
where the homeomorphism problem is difficult, but mathematically settled~\cite{Morgan14GeometrizationConjecture}. 
More precisely, in this article we focus on the family of \emph{Turaev-Viro invariants} $\tv_{r,q}$, parameterised by integers $r$ and $q$, which are amongst the most powerful invariants for 
$3$-manifolds~\cite{turaev92-invariants}. Similarly to the Jones polynomial for 
knots, they derive from quantum field theory but can be computed by purely 
combinatorial means. Algorithms to compute these invariants are implemented for 
$3$-manifolds represented by triangulations (see the software package 
\emph{Regina}~\cite{regina}) and special spines (see the software 
\emph{Manifold Recogniser}~\cite{matveev03-algms,recogniser}), and they play a 
key role in enumerating $3$-manifolds of bounded topological complexity (an
analogue to the famous knot tables)~\cite{burton07-nor7,matveev03-algms}. 

\paragraph{Complexity and existing algorithms.} Turaev-Viro invariants are defined as exponentially large state-sums over combinatorial data---so called {\em admissible 
colourings}---defined on the edges and triangles of a triangulated $3$-manifold. A naive exponential time algorithm to compute them consists of enumerating all potential colourings of the triangulation, and sum up their weights.

Recently, new algorithms and techniques have been introduced to improve performance. In~\cite{Burton15TuraevViro}, Burton and the authors introduce a fixed parameter tractable algorithm for computing any Turaev-Viro invariant, using the \emph{treewidth of the dual graph of the triangulation} as parameter. Algorithms can be improved further by pruning the search space for admissible colourings for some of the Turaev-Viro invariants~\cite{Maria15TVAlgos}. While significantly improving both the practical and theoretical complexity of the computation, these algorithms are still exponential in both time and space complexity.

Moreover, it is known that, on a $3$-manifold represented by a triangulation, computing the Turaev-Viro invariant $\tv_{4,q}$ is 
\#P-hard~\cite{Burton15TuraevViro,kirby04-nphard}. Hence, the invariant $\tv_{4,q}$ is of particular interest for complexity theory. In this article we focus on this particular invariant. 

\paragraph{Our contribution.} In this article we introduce a \emph{fixed parameter tractable algorithm} to compute $\tv_{4,q}$ on a triangulation $\tri$, using the {\em rank of the first homology group} as parameter for the algorithm: 

\begin{thm}
  Let $\tri$ be a $1$-vertex\footnote{Having a $1$-vertex triangulation is a 
  rather weak restriction, see the discussion in Section~\ref{sec:fpt}.} 
  $n$-tetrahedra triangulation of a $3$-manifold with first Betti number
  $\beta_1 (\tri,\mathbb{Z}_2)$, then there exists an algorithm to compute 
  $\tv_{4,q} (\tri)$, $q \in \{1,3\}$,
  with running time
  $O(2^{\beta_1 (\tri, \mathbb{Z}_2)} n^3)$ and $O(n^2)$ space requirements.
\end{thm}

The algorithm interprets $\tv_{4,q}$ as a sum over the weights of a family of 
embedded surfaces within the triangulation. We show that each of these surfaces 
can be assigned a $1$-cohomology class $\theta \in H^1 (\tri, \mathbb{Z}_2)$
which defines its weight, up to a sign. We show that, for each 
such $\theta$, the set of surfaces associated to $\theta$ (with all of its 
members necessarily having the same weight, up to a sign) can be efficiently 
described as the solution space of a linear system of size $O(n)$.
The sign of the weight of a fixed surface in this space is determined by the 
parity of the number of certain surface pieces (octagons, or so-called ``almost 
normal'' surface pieces). We show that the number of surfaces in the solution 
space with a fixed parity equals the number of zeroes of a quadratic form over 
$\mathbb{Z}_2$ on this space. Following the theory of quadratic forms over 
$\mathbb{Z}_2$ this quadratic form can be transformed into standard form 
yielding its number of zeroes. This is all we need to compute the sum of weights
over all surfaces corresponding to $\theta$. Summing over all cohomology classes
leads to the algorithm.

\paragraph{Discussion of the parameter.} As mentioned above, the state of the 
art algorithm to compute $\tv_{r,q}$ is fixed parameter tractable in the 
\emph{treewidth} of the dual graph of the triangulation. 
The benefits of our new fixed parameter tractable algorithm for the special 
case of $r=4$, are thus entirely due to the first Betti number as the 
parameter in use, 
which is superior to the treewidth in several key aspects: 

\begin{description}
  \item[Parameter is topological] Treewidth is a property of the 
    triangulation in use. This means that even triangulations of very simple 
    manifolds, such as the $3$-sphere or other lens spaces, can be represented 
    by an input triangulation of arbitrarily high treewidth.
    The first Betti number is a topological invariant of the underlying 
    manifold and thus {\em independent of the choice of triangulation}. 
    Eliminating such a dependence on the combinatorial structure of a
    triangulation is highly desirable in the field of computational topology.

    In fact, to our knowledge, this is the first non-trivial fixed parameter
    tractable algorithm of a problem in $3$-manifold topology using a 
    topological parameter. Note that, for algorithmic problems dealing with surfaces,
    similar results exist. For instance, graph embeddability is known 
    to be fixed parameter tractable in the genus of the surface 
    \cite{Mohar99FPTGraphEmbeddings}.
  \item[Many inputs have small $\beta_1$ parameter] Bounded treewidth is a condition which
    is closed under minors. It thus follows from standard results in forbidden 
    minor theory and the theory of triangulations that, for a given number of 
    tetrahedra $n$, the number of triangulations of bounded treewidth is bounded from above by an exponential function. On the other hand, the number of $3$-manifolds which can be 
    triangulated with $\leq n$ tetrahedra grows 
    super-exponentially fast in $n$. Thus many $3$-manifolds only have few 
    triangulations with small treewidth. In contrast, there exist extremely 
    large classes of $3$-manifolds with uniformly bounded first Betti number $\beta_1$, 
    and for each one of them {\em all} triangulations necessarily must 
    have bounded first Betti number.

    Hence, despite the computation of $\tv_{4,q}$ being \#P-hard, the algorithm presented in this article has polynomial complexity for very large families of inputs. 
  \item[Parameter is efficiently computable] Given a graph, computing
    its treewidth is NP-complete~\cite{arnborg87-embeddings}. The problem is known to be fixed parameter 
    tractable in the natural parameter~\cite{bodlaender96-linear}, but in practice this algorithm 
    is not the method of choice. Thus, in practice, it can be difficult to 
    decide whether a given triangulation has a small treewidth. In contrast, 
    the running time of computing the first Betti number of a triangulation is
    a small polynomial, regardless of the size of the Betti number.
\end{description}

Furthermore, while the space requirements of the 
treewidth algorithm is exponential in the parameter, our 
algorithm only uses {\em quadratic space} in the input size, regardless of the 
size of the parameter.

\paragraph{Structure of the article.}
The paper is organised as follows. 
After going over some important concepts used
in the article, we describe the FPT algorithm for $\tv_{4,q}$ in three steps. 
In Section~\ref{sec:surfaces}, we start by describing embedded surfaces defined
by admissible colourings, and show that their weights can be interpreted as a  
function of their topological and combinatorial properties. In doing so, we split
the sum of weights defining $\tv_{4,q} (\tri)$ for a triangulation $\tri$ by 
grouping colourings by associated $1$-cohomology classes.
In Section~\ref{ssec:poly}, we introduce a polynomial time 
algorithm to compute the weight participation of a set of colourings assigned to
a given $1$-cohomology class. 
In Section~\ref{ssec:fpt}, the FPT algorithm then finally follows by running this 
procedure on all of the $2^{\beta_1 (\tri,\mathbb{Z}_2)}$ cohomology classes.

In Section~\ref{sec:counting} we discuss implications of the algorithm from a 
complexity theoretic point of view. Recall that $\tv_{4,q}$ is known to be 
\#P-hard due to work by Kirby and Melvin \cite{kirby04-nphard} and a slight 
adjustment by Burton and the authors \cite{Burton15TuraevViro}. Using the 
structure of our algorithm we show that $\tv_{4,q}$ is not harder than counting 
(see Section~\ref{sec:counting}), thus further bounding the complexity of 
computing $\tv_{4,q}$ from above.

In Section~\ref{sec:expts} we focus on the benefits of the new algorithm 
for research in computational topology---which strongly depend on the power of 
$\tv_{4,q}$ to distinguish between manifolds with equal homology. 
We provide theoretical
and experimental evidence that our algorithm, in combination with integral 
homology,\footnote{Integral homology groups, i.e., 
homology groups with integer 
coefficients, are strictly more powerful than homology groups with finite field 
coefficients, but can still be computed in polynomial time.} provides an 
efficient tool to distinguish between almost twice as many manifolds as integral homology on its own.

\section{Background}
\label{sec:bg}

\paragraph{Manifolds and generalised triangulations.} 
Throughout this article closed $3$-manifolds are given in the widely
used form of \emph{generalised triangulations}. Generalised triangulations
are more general than simplicial complexes and can encode a wide range of 
manifolds, and very complex topologies, with very few tetrahedra.

More precisely, a generalised triangulations $\tri$ of a (closed) $3$-manifold $\M$ 
is a collection of $n$ abstract tetrahedra $T = \{ \Delta_1,\ldots,\Delta_n \}$ 
together with $2n$ {\em gluing maps} identifying their $4n$ triangular faces 
in pairs, such that the underlying topological space is homeomorphic to $\M$.
An equivalence class of vertices, edges, or triangles of $\Delta_i$, $ 1 \leq i \leq n$,
identified under the gluing maps, is referred to as a single vertex, edge, or triangle
of $\tri$. We denote by $V$, $E$, and $F$ the sets of such vertices, edges, and triangles, 
respectively, of $\tri$. It is common in practical applications to have 
\emph{one-vertex triangulations} where all $4n$ vertices of $\Delta_i$, $ 1 \leq i \leq n$, 
are identified to one point. The number of tetrahedra $n$ of $\tri$ is often referred to as 
the {\em size} of the triangulation.

Since, by construction, every $n$-tetrahedra $v$-vertex closed $3$-manifold 
triangulation $\tri$ must have $2n$ triangles, and every closed $3$-manifold has 
Euler characteristic zero, it follows that $\tri$ must have $n+v$ edges.

We refer the reader to~\cite{jaco03-0-efficiency} for more details on generalised triangulations. 

\paragraph{Homology, cohomology and Betti numbers.} 
We use basic facts about the homology and cohomology groups of a (triangulated)
$3$-manifold $\tri$ with $\Z_2$-coefficients. We denote by $\beta_1 (\tri, \Z_2)$ 
the first Betti number of $\tri$, i.e., the rank of the first homology group 
with $\Z_2$ coefficients.

See Appendix~\ref{app:homology} for a brief, and \cite{hatcher02-algebraic} for a 
longer introduction to (co-)homology theory.

\paragraph{Turaev-Viro invariants.} 
Turaev-Viro invariants are part of a larger group of {\em invariants 
of Turaev-Viro type} $\tv_r$, parameterised by an integer $r \geq 3$. We 
first define this more general group of invariants before having a closer 
look at the original {\em Turaev-Viro invariants} $\tv_{r,q}$, which also 
depend on a second integer $0 < q < 2r$ coprime to $r$.

Let $\tri$ be a generalised triangulation of a closed $3$-manifold $\M$,
let $r \geq 3$, be an integer, and let $I = \{0, 1/2, 1, 3/2, \ldots, (r-2)/2\}$.
A \emph{colouring} of $\tri$ is defined to be a map $\theta\co E \to I$
from the edges of $\tri$ to $I$. A colouring $\theta$ is \emph{admissible} if, for each triangle of 
$\tri$, the three edges $e_1$, $e_2$, and $e_3$ bounding the triangle satisfy 
the 
\begin{itemize}
  \item \emph{parity condition} $\theta(e_1)+\theta(e_2)+\theta(e_3)\in \Z$;
  \item \emph{triangle inequalities} $\theta(e_i) \leq \theta(e_j) + 
    \theta(e_k)$, $\{i,j,k\} = \{ 1,2,3\}$; and 
  \item \emph{upper bound constraint} $\theta(e_1)+\theta(e_2)+\theta(e_3)\leq 
    r-2$.
\end{itemize}
The set of such admissible colourings is denoted by $\adm (\tri,r)$.

For each admissible colouring $\theta$, and for each vertex $w \in V$, edge 
$e \in E$, triangle $f \in F$ or tetrahedron $t \in T$ we define 
\emph{weights} $|w|_{\theta}, |e|_{\theta}, |f|_{\theta}, 
|t|_{\theta} \in \C$. The 
weights of vertices are constant, and the weights of edges, triangles and 
tetrahedra only depend on the colours of edges they are incident to.
Using these weights, we define the \emph{weight of the colouring} to be
\begin{equation}
|\tri|_{\theta} =
    \prod_{w \in V} |w|_{\theta} \times
    \prod_{e \in E} |e|_{\theta} \times
    \prod_{f \in F} |f|_{\theta} \times
    \prod_{t \in T} |t|_{\theta},
\end{equation}

\emph{Invariants of Turaev-Viro type} of $\tri$ are defined as sums of the 
weights of all admissible colourings of $\tri$, that is
$\tv_{r}(\tri) = \sum_{\theta \in \adm(\tri,r)} |\tri|_{\theta}$.
In~\cite{turaev92-invariants} Turaev and Viro show that, whenever the weighting 
system satisfies some identities, $\tv_{r}(\tri)$ is a topological invariant of 
the manifold; that is, if $\tri$ and $\tri'$ are generalised triangulations of 
the same closed 3-manifold $\M$, then $\tv_{r}(\tri) = \tv_{r}(\tri')$.

More specifically, in Appendix~\ref{app:weight} we give the definition of the 
weights in the case of the original Turaev-Viro invariants, which we use in this article. 
These not only depend on $r$ but also on a second integer $0 < q < 2r$, with 
$\gcd(r,q) = 1$. We denote the corresponding invariant by $\tv_{r,q}$. 

We continue to write $| \cdot |_{\theta}$ for the weights where $r$ and/or $q$ are
given from context or the statement holds in more generality. 

For an $n$-tetrahedra triangulation $\tri$ with $v$ vertices 
there is a simple backtracking algorithm to compute 
$\tv_{r,q}(\tri)$ by testing the $(r-1)^{v+n}$ possible colourings for 
admissibility and computing their weights. The case $r=3$ can however be 
computed in polynomial time, due to a connection between $\adm(\tri,3)$ and 
cohomology, see 
\cite{Burton15TuraevViro,matveev03-algms}. The case $r=4$, which we study in this article, 
is \#P-hard to compute in general~\cite{Burton15TuraevViro,kirby04-nphard}.

\paragraph{Quadratic forms over $\Z_2$.}
A \emph{quadratic form over $\Z_2$} is a form $\phi\colon \Z_2^k \to \Z_2$ satisfying 
$\phi\colon v \mapsto v^T R v$, for a fixed $(k \times k)$-matrix $R$.
Two quadratic forms $\phi\colon v \mapsto v^T R v$ and $\phi\colon v \mapsto v^T S v$ 
are called {\em equivalent}, if there exists a matrix $C \in \operatorname{GL}(k,\Z_2)$ such that 
$R = C^T S C$. Equivalent quadratic forms have the same number of zeroes.
A quadratic form in $k$ indeterminants is called {\em degenerate} whenever 
it is equivalent to a quadratic form depending on less than $k$ indeterminants.
Otherwise it is called {\em non-degenerate}.

Note that the theory of quadratic forms over fields of characteristic two
is significantly distinct from the general theory.
Most quadratic forms over fields of characteristic two can not be represented by symmetric 
matrices and thus are not diagonalisable. Moreover, a diagonalisable quadratic form over a field of characteristic two must be 
equivalent to a quadratic form in (at most) one indeterminant. 

In what follows, whenever we work with a fixed quadratic form $\phi$ we 
assume it is represented by an upper triangular matrix -- which is always possible.

\begin{lemma}[see Theorem 6.30 in~\cite{Lidl1997Finite} ]
  Let $\phi$ be a (possibly degenerate) quadratic form over $\Z_2$. Then
  $\phi$ is equivalent to the direct sum of the all-zero quadratic form in 
  $k-\ell$ indeterminants (admitting $2^{k-\ell}$ zeroes), and one of the following 
  non-degenerate quadratic forms in $\ell \leq k$ indeterminants.
\[
\left\{
\begin{array}{lll}
  \text{If $\ell$ is odd}, & 
    x^T \tilde{R} x = x_1x_2 + x_3x_4 + \ldots + x_{\ell-2}x_{\ell-1} + x_{\ell}^2,
  & \text{admitting $2^{\ell-1}$ zeros}. \\
  \text{If $\ell$ is even}, & 
    x^T \tilde{R} x = x_1x_2 + x_3x_4 + \ldots + x_{\ell-1}x_{\ell}, 
&
\text{admitting $2^{\ell-1} + 2^{\frac{\ell}{2} -1}$ zeros}, \\

  \text{or} & 
    x^T \tilde{R} x = x_1x_2 + x_3x_4 + \ldots + x_{\ell}x_{\ell} + x_{\ell-1}^2 + x_{\ell}^2.
  &
\text{admitting $2^{\ell-1} - 2^{\frac{\ell}{2} -1}$ zeros}. \\
\end{array}\right.
\]
\label{lem:quadform}
\end{lemma}

Given a quadratic form $\phi \colon \mathbb{Z}_2^k \to \mathbb{Z}$ in indeterminants $x_1, x_2, \ldots , x_k$
we can determine $\ell$ and reduce $\phi$ to either one of the three forms of 
Lemma~\ref{lem:quadform} in polynomial time following the constructive 
proof of Theorem 6.30 in~\cite{Lidl1997Finite}. 

The proof repeatedly splits $\phi$ into blocks of form $x_1 x_2$ and a 
new quadratic form $\phi'$ in indeterminants $x_3, x_4, \ldots , x_k$
(this step is described in detail in Lemma 6.29 in~\cite{Lidl1997Finite}). 
One such splitting step requires a constant number of
variable relabelings and sparse variable substitutions, and is able to 
detect and handle degeneracies. The distinction between the three cases is made in the 
last step when $k \leq 2$, where at most $2^3 = 8$ possible cases have to be considered.

The overall number of zeroes of $\phi$
follows by multiplying the number of zeroes of the non-degenerate part
by $2^{k-\ell}$ (note that the all-zero quadratic form never
evaluates to $1$).

\section{Surfaces interpretation and weight system}
\label{sec:surfaces}

\paragraph{Surface interpretation for $r=4$.} An admissible colouring for the triangulation $\tri$ may be interpreted as a \emph{spinal surface} embedded within $\tri$, that is, a surface intersecting the faces of $\tri$ transversally, whose intersection with each triangle is a collection of straight line segments, and with each tetrahedron is a collection of topological disks. To define a spinal surface $S_{\theta}$ from an admissible colouring $\theta$, interpret $\theta(e)$ as half the number of times $S_\theta$ intersects $e$. The admissibility constraints ensure that there is a well-defined and unique spinal surface with such an intersection pattern: all such surfaces may be classified for every value of $r$~\cite{frohman08-quantum,maria16-normcurves-yrf}, and all intersection patterns for $r=4$ are pictured in Figure~\ref{fig:all-col}. 

\begin{figure}
\centering
\includegraphics[width=15cm]{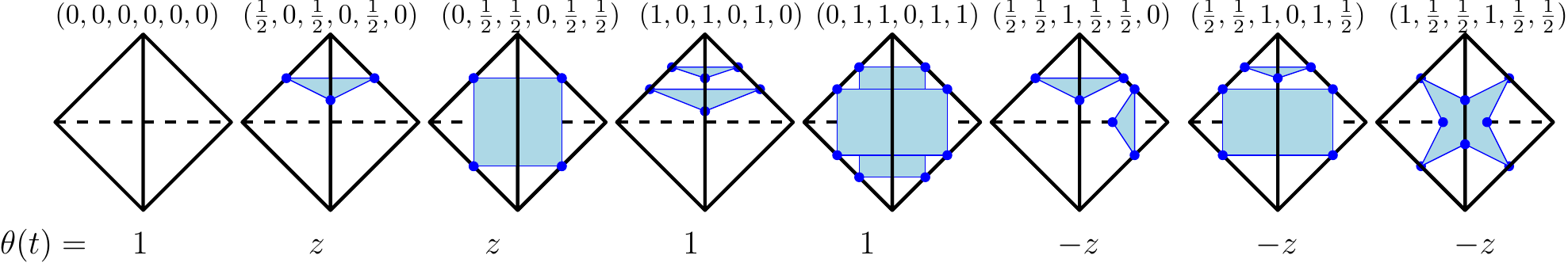}
\caption{All tetrahedron intersection patterns of an admissible colouring for $r=4$, their edge colours (top) and tetrahedron weights for $\tv_{4,q}$ (bottom), where 
$z \in \{ \sqrt{2}, -\sqrt{2} \}$ depending on $q$.}
\label{fig:all-col}
\end{figure}

The three leftmost intersection patterns in Figure~\ref{fig:all-col} are the ones for $r=3$, where edge colours belong to $\{0,\frac{1}{2}\}$. The admissibility constraints ensure that $\adm(3,\tri)$ is in one-to-one correspondence with
the set of $\Z_2$ $1$-cycles of $\tri$~\cite{Maria15TVAlgos,matveev03-algms}. Additionally, if $\tri$ is a $1$-vertex triangulation, the only $1$-coboundary of $\tri$ is the trivial one, and $\adm(3,\tri)$ is in one-to-one correspondence with $\Hom^1(\tri,\Z_2)$, the $1$-cohomology group of $\tri$ with $\Z_2$ coefficients. We use this last fact in Section~\ref{sec:fpt}.

In the following, we talk about an admissible colouring $\theta$ and its spinal surface $S_\theta$ interchangeably, and in particular talk about the \emph{Euler characteristic} of a colouring $\theta$ defined as $\chi(S_{\theta})$. For $\theta \in \adm(\tri,3)$, we also talk indifferently of the admissible colouring and the corresponding $1$-cocycle in $\Z_2$-cohomology.

\paragraph{Weight system.} Before we can describe the algorithm, we first need to have a closer look at the weights of edges, triangles, and tetrahedra defined in Section~\ref{sec:bg} for the case $r=4$, and $q$ such that $1 \leq q \leq 8$, $\gcd(4,q)=1$.

First, note that 
the values of the quantum integers $[k]$, $0\leq k \leq 4$, are given by
\[
[1] = [3] = 1 \text{ and } [4] = 0 \ \text{for all $q$, } \ \ \text{ and }  
[2] = \left\{ \begin{array}{rl} \sqrt{2} & \text{if $q \in \{1,7\},$}\\
                                - \sqrt{2} & \text{if $q \in \{3,5\}$}.
              \end{array} \right.
\]
For the remainder of this article we define $z := -[2]$, with $z \in \{\sqrt{2}, -\sqrt{2} \}$ depending on the integer $q$.

We study the weights of faces, according to their colours. 
Let $\theta \in \adm(\tri,4)$, that is, $\theta$ colours the edges of $\tri$
with colours $0$, $\frac{1}{2}$, and $1$, such that---up to permutation---the three edges of each triangle are coloured  $(0,0,0)$, 
$(0,\frac{1}{2},\frac{1}{2})$, $(\frac{1}{2},\frac{1}{2},1)$, and $(0,1,1)$.
For a vertex $w$, an edge $e$, and a triangle $f$, we have
\[
|w|_\theta = \frac{1}{4}, \ \ \  
|e|_{\theta} = \left \{ \begin{array}{rl} 
                        1   & \textrm{if } \theta(e) = 0 \\
                        z & \textrm{if } \theta(e) = \frac{1}{2} \\
                        1   & \textrm{if } \theta(e) = 1 
                        \end{array} \right. \ \ \ \text{and} \ \ \  
|f|_{\theta} = \left \{ \begin{array}{rl@{\,\,}l@{}l@{}l@{}l@{}l}
                1 &\textrm{if } f \textrm{ is coloured} &(&0,&0,&0&) \\
                z^{-1} &\ditto &(&0,&\frac{1}{2},&\frac{1}{2}&) \\
                -z^{-1} &\ditto &(&\frac{1}{2},&\frac{1}{2},&1&) \\
                1 &\ditto &(&0,&1,&1&) .
              \end{array} \right .
\]
Tetrahedron weights are presented in Figure~\ref{fig:all-col}. 

Looking at the definition of Turaev-Viro type invariants (see Section~\ref{sec:bg}) and the observation made above, we deduce that $\tv_{4,q}$ is a Laurent polynomial in $z$, evaluated at $\pm \sqrt{2}$. To see that other values of $z$ can not lead to a topological invariant, consider two tetrahedra coloured $(\frac12, 0, \frac12, 0, \frac12, 0)$ joined along the zero coloured triangle $t$ (see Figure~\ref{fig:handle}). The two tetrahedra and the common triangle contribute a factor of $z^{-1}$ to this colouring. Performing a $2$-$3$-move (i.e., replacing two tetrahedra joined along a triangle by three tetrahedra joined along an edge) across $t$, yields three tetrahedra joined along a common edge $e$. Keeping the colouring on all boundary edges fixed, $e$ can be coloured $0$ or $1$ leading to two valid colourings. In each case, the three tetrahedra weights multiplied with the three internal triangle weights and the internal edge weight contribute a factor of $z^{-3}$ to the colouring. Since the $2$-$3$-move does not change the topology of the triangulation, the sum of the weights of the two new colourings must equal the weight of the original colouring. Hence, $z^{-1} = z^{-3} + z^{-3}$ and thus at most $z \in \{\pm \sqrt{2} \}$. However, we know from above that both solutions give rise to a topological invariant.

\begin{figure}
  \centerline{\includegraphics[width=10cm]{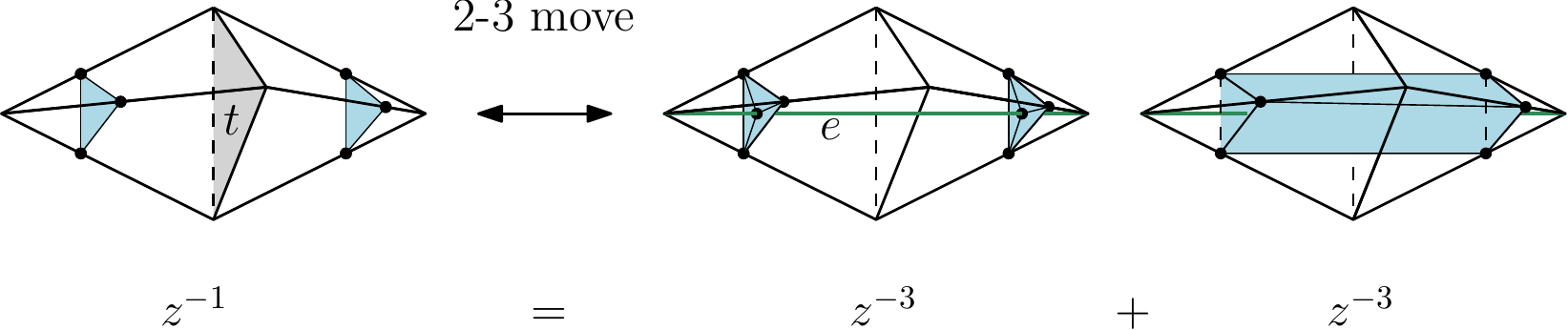}}
  \caption{$2$-$3$-move performed on two coloured tetrahedra, yielding two 
    possible colourings with equal boundary pattern: for the state sum 
    to define an invariant, $z$ must satisfy $z^2 = 2$. 
    \label{fig:handle}}
\end{figure}

By Matveev \cite{matveev87-transformations}, and independently by Piergallini \cite{Piergallini88StandardMoves}, we know that any two $1$-vertex triangulations of a $3$-manifold are connected by a sequence of $2$-$3$-moves and their inverses. There are $28$ more constellations (up to symmetry) of how two coloured tetrahedra from the list in Figure~\ref{fig:all-col} can meet along a triangle. Performing a $2$-$3$-move on one of them gives rise to an equivalent condition ($z^2 = 2$), the other $27$ do not impose any restrictions at all. This defines a very basic (although slightly lengthy) proof of the topological invariance of $\tv_{4,q}$.

\smallskip
For the rest of the article, we will omit the constant vertex weight $1/4$ when defining colouring weights and the Turaev-Viro invariant. 
Note that in particular we can follow from the above calculations that all computations are done within the extension ring $\Z[\sqrt{2}]$, where arithmetic operations are constant time representing $\sqrt{2}$ symbolically.

\paragraph{Topological interpretation of weights for $\tv_{4,q}$.} 
We interpret the weights of colourings in terms of the Euler characteristic of associated surfaces. For admissible colourings with edge colours $\{0,\frac{1}{2}\}$, we have: 

\begin{lemma}
\label{lem:adm3}
  Let $\theta \in \adm(\tri,4)$ such that no edge of $\tri$ is coloured with
  $1$, and let $S_{\theta}$ be the surface associated with $\theta$. 
  Then $\theta \in \adm(\tri,3)$, and
  $$ |\tri|_{\theta} = z^{\chi (S_\theta)} $$
  where $\chi$ denotes the Euler characteristic.
\end{lemma}

\begin{proof}
  First note that $\theta \in \adm(\tri,4)$ with no edge coloured by $1$ 
  implies that all triangles are coloured $(0,0,0)$ or 
  $(0,\frac{1}{2},\frac{1}{2})$ (up to symmetry) and thus $\theta \in \adm(\tri,3)$.

  The proof is a direct corollary of the face weights listed above. 
  Let $S_{\theta}$ have $m_0$ vertices, $m_1$ edges, $m_{\triangle}$ triangles 
  and $m_{\square}$ quadrilaterals or, in other terms, let $m_0$ be the number of 
  edges of $\tri$ that are coloured $\frac{1}{2}$ by $\theta$, $m_1$ be the number of $(0,\frac{1}{2},\frac{1}{2})$ coloured triangles (up to symmetry), 
  $m_{\triangle}$ be the number of $(\frac{1}{2},0,\frac{1}{2},0,\frac{1}{2},0)$ coloured tetrahedra (up to symmetry), 
  and $m_{\square}$ be the number of $(0,\frac{1}{2},\frac{1}{2},0,\frac{1}{2},\frac{1}{2})$ coloured 
  tetrahedra (up to symmetry). All other faces must be zero-coloured and hence have weight $1$. 

  It follows that we have for the product of all weights
  $$ |\tri|_{\theta} = z^{m_0} \cdot z^{-m_1} \cdot z^{m_{\triangle}} \cdot z^{m_{\square}} = 
  z^{\chi (S_\theta)} $$
  where $\chi$ denotes the Euler characteristic of the surface $S_{\theta}$.
\end{proof}

For a colouring $\hat{\theta} \in \adm(\tri,4)$, we define its \emph{reduction} $\theta$ satisfying:
\begin{equation}
\label{eq:red}
  \text{For every edge $e \in E$, } \theta(e) = \hat{\theta}(e) - \lfloor \hat{\theta}(e) \rfloor .
\end{equation}
The reduction of an admissible colouring of $\adm(\tri,4)$ is an admissible colouring of $\adm(\tri,3)$. 

\begin{lemma}
  \label{lem:octopi}
  Let $\hat{\theta} \in \adm(\tri,4)$ and let $\theta$ be the reduction of 
  $\hat{\theta}$. Then 
  $$ |\tri|_{\hat{\theta}} = (-1)^{\alpha} |\tri|_{\theta} = (-1)^{\alpha} z^{\chi (S_\theta)}, $$
  where $\alpha$ denotes the number of tetrahedra coloured $(1,\frac{1}{2},\frac{1}{2},1,\frac{1}{2},\frac{1}{2})$
  in $\hat{\theta}$ (up to symmetry), i.e., the number of octagons in $S_{\hat{\theta}}$.
\end{lemma}

\begin{proof}
  We want to express the weight of $\hat{\theta}$ in terms of the weight of its reduction $\theta$. 

  Following the study of weights for $\tv_{4,q}$ above, the only face colourings of $\hat{\theta}$ whose weight changes in $\theta$ are: the triangle $(\frac{1}{2},\frac{1}{2},1)$, and the tetrahedra $(\frac{1}{2},\frac{1}{2},1,\frac{1}{2},\frac{1}{2},0)$, $(\frac{1}{2},\frac{1}{2},1,0,1,\frac{1}{2})$ and $(1,\frac{1}{2},\frac{1}{2},1,\frac{1}{2},\frac{1}{2})$. Their weights differ only by a factor of $(-1)$. Let $\gamma$ be the total number of those faces whose weights with $\hat{\theta}$ change in the reduction. We have that $|\tri|_{\hat{\theta}} = (-1)^{\gamma} |\tri|_{\theta}$.

  Note that $(1,\frac{1}{2},\frac{1}{2},1,\frac{1}{2},\frac{1}{2})$ contains four, $(\frac{1}{2},\frac{1}{2},1,\frac{1}{2},\frac{1}{2},0)$ 
  and $(\frac{1}{2},\frac{1}{2},1,0,1,\frac{1}{2})$ contain two, and all other tetrahedra types
  contain zero triangles of type $(\frac{1}{2},\frac{1}{2},1)$. Moreover, every triangle is 
  contained in two tetrahedra. If there are $\alpha$ tetrahedra of octagon type 
  $(1,\frac{1}{2},\frac{1}{2},1,\frac{1}{2},\frac{1}{2})$, $\lambda$ of type $(\frac{1}{2},\frac{1}{2},1,\frac{1}{2},\frac{1}{2},0)$, and $\mu$ of
  type $(\frac{1}{2},\frac{1}{2},1,0,1,\frac{1}{2})$, we have that 
  $$\gamma = \alpha + \lambda + \mu + (4\alpha + 2\lambda + 2\mu)/2 = 3\alpha + 2\lambda + 2\mu $$
  and thus $(-1)^{\gamma} = (-1)^{\alpha}$, and, by virtue of Lemma~\ref{lem:adm3},
  $$ |\tri|_{\hat{\theta}} = (-1)^{\alpha} |\tri|_{\theta} = (-1)^{\alpha} z^{\chi (S_\theta)}. $$
\end{proof}

\section{Fixed parameter tractable algorithm in $\beta_1$ for $\tv_{4,q}$}
\label{sec:fpt}

Throughout this section we assume that $\tri$ is a $1$-vertex triangulation of a closed $3$-manifold. This is a reasonable assumption since input triangulations in computational $3$-manifold topology are typically presented in this form. Moreover, note that, given an arbitrary triangulation of a closed $3$-manifold, there exists a polynomial time algorithm to construct a $1$-vertex triangulation $\tri'$ with $|\tri'| \leq |\tri|$~\cite{burton14-crushing-dcg,burton12-unknot}.\footnote{The procedure may fail in the rare case of triangulations containing two-sided projective planes.} 
	
In this section we present a \emph{fixed parameter tractable algorithm} (FPT) to compute $\tv_{4,q} (\tri)$, for any $1 \leq q \leq 8$, $\gcd(4,q)=1$, which runs in polynomial time in the size of $\tri$ as long as the first
Betti number of $\tri$ is bounded. More precisely, the algorithm has running
time $O(2^{\beta_1 (\tri, \mathbb{Z}_2)}\cdot n^3)$.

\subsection{Polynomial time algorithm at a cohomology class.} 
\label{ssec:poly}
Let $\theta$ be an admissible colouring of $\adm(\tri,3)$ (i.e., we fix a $1$-cohomology class). We define:
$$ A_{\theta} = \{ \hat{\theta} \in \adm(\tri,4) \,\,|\,\, \hat{\theta} \textrm{ reduces to } \theta \} \ \ \ \ \text{and} \ \ \ \ \tv_{4,q}(\tri,[\theta]) := \sum_{\hat{\theta} \in A_\theta} |\tri|_{\hat{\theta}} 
$$ 
to be the set of colourings reducing to $\theta$ via Equation~\ref{eq:red}, and the sum of their weights respectively. By virtue of Lemma~\ref{lem:octopi}, the weights $|\tri|_{\hat{\theta}}$ of the sum are all equal, up to a sign, to $z^{\chi(S_\theta)}$.

This partial sum of the Turaev-Viro invariant is the \emph{Turaev-Viro invariant at a cohomology class}~\cite{turaev92-invariants}. We present a polynomial time algorithm to compute $\tv_{4,q}(\tri,[\theta])$ at a given cohomology class $[\theta]$.

\paragraph{Characterisation of the space of colourings $A_{\theta}$.} Given $\theta \in \adm(\tri,3)$, we partition the set of edges $E$
of $\tri$ into three groups $E_0$, $E_1$, and $E_2$: 
\begin{itemize}
\item[-] $E_0$ contains all edges 
coloured by $\frac{1}{2}$ in $\theta$, 
\item[-] $E_1$ contains all edges coloured $0$ which occur in at least one triangle of type $(0,0,0)$, and 
\item[-] $E_2$ contains all edges 
coloured $0$ which only occur in triangles of type 
$(0,\frac{1}{2},\frac{1}{2})$. 
\end{itemize}
We characterise the space of colourings $A_{\theta}$ as the solution of a set of linear equations. 

By definition, the edges in $E_0$ are exactly the ones coloured $\frac{1}{2}$ by all colourings $\hat{\theta} \in A_{\theta}$. 
Every admissible colouring $\hat{\theta} \in A_{\theta}$ must colour triangles of type $(0,0,0)$ in $\theta$ by either $(0,0,0)$ or $(0,1,1)$, up to permutation. Hence, such a triangle $\{e_1,e_2,e_3\}$ is admissible if and only if $\hat{\theta} (e_1) + \hat{\theta} (e_2) + \hat{\theta} (e_3) = 0 \mod 2$. Considering $\hat{\theta} (e)$, $e \in E_1$, as an element of $\Z_2$, all possible colourings of these triangles can be described by a homogeneous linear system over $\Z_2$, that is, the incidence matrix of edges in $E_1$ and triangles of 
type $(0,0,0)$ in $\theta$.

Observe that every solution of this system can be extended to an admissible
colouring $\hat{\theta} \in A_{\theta}$ by assigning colour $0$ to all edges in 
$E_2$. Indeed, all triangles of type $(0,0,0)$ in $\theta$ are now of type $(0,0,0)$ or $(1,1,0)$ in $\hat{\theta}$, and all triangles of type 
$(\frac{1}{2},\frac{1}{2},0)$ in $\theta$ are now of type 
$(\frac{1}{2},\frac{1}{2},0)$ or $(\frac{1}{2},\frac{1}{2},1)$.

Finally, by definition of the set $E_2$, every assignment of colours to the edges $E_0 \cup E_1$, satisfying the conditions above, gives rise to $2^{|E_2|}$ admissible
colourings given by all possible $\{0,1\}$ assignments of colours to edges in $E_2$. Take a moment to verify that no such assignment of colours can result in a non-admissible triangle colouring.

It follows that the set $A_{\theta}$ can be described as a subspace in 
$\mathbb{Z}^{|E_1| + |E_2|}$ where the first $|E_1|$ coordinates are associated 
with the edges in $E_1$ and the last $|E_2|$ coordinates are associated with 
the edges in $E_2$ (the edges in $E_0$ are always coloured $\frac{1}{2}$ and 
thus need no explicit description). A basis of the subspace is given by a 
basis of the solution space of the linear system $\{ b_1, b_2, \ldots , b_m\}$,
concatenated with the standard basis on the last $|E_2|$ coordinates
$\{ d_1, d_2, \ldots , d_{|E_2|} \}$. The subspace naturally decomposes into two blocks of size $m$ and $|E_2|$.

\paragraph{Evaluation of $\tv_{4,q}(\tri,[\theta])$.} Hence, using the characterisation above, we can efficiently compute the cardinality of $A_{\theta}$. Furthermore, we know that all colourings in $A_{\theta}$ have the same weight, up to a sign which only depends on the parity of the number of octagons of a colouring.  

Thus, it remains to show
that we can determine the number of admissible colourings in $A_{\theta}$ with an even number
of octagons in polynomial time.

\medskip
With $\theta \in \adm(\tri,3)$ fixed, the only tetrahedra which can be of type
$(1,\frac{1}{2},\frac{1}{2},1,\frac{1}{2},\frac{1}{2})$ in $\hat{\theta} \in 
A_{\theta}$ are the ones of type 
$(0,\frac{1}{2},\frac{1}{2},0,\frac{1}{2},\frac{1}{2})$ in $\theta$. Denote 
these tetrahedra by $t_1, \ldots , t_s$, and denote their opposite $0$ coloured edges by $x_i$, $y_i$, $1 \leq i \leq s$. Considering the $0$ or $1$ colour of an edge as elements of $\Z_2$, the parity of the number of 
octagons in $\hat{\theta}$ is now given by the quadratic form: 
\begin{equation}
  \label{eq:quad}
  \sum \limits_{i = 1}^{s} \hat{\theta} (x_i) \hat{\theta} (y_i) \in 
  \mathbb{Z}_2 
\end{equation}
which can be represented by an upper triangle matrix 
$Q = \mathbb{Z}_2^{(|E_1| + |E_2|)\times(|E_1| + |E_2|)}$ by setting 
$Q_{i,j} = 1$, $i \leq j$, if and only if a specific pair of edges occurs in
an odd number of terms in Equation~\ref{eq:quad} (note that, in a generalised triangulation, two edges may appear as opposite edges in more than one tetrahedron). 

Every colouring is given by a linear combination of vectors 
$b_i$, $1 \leq i \leq m$, and $d_j$, $1 \leq j \leq |E_2|$, that is, by
a vector in $\mathbb{Z}_2^{|E_1| + |E_2|}$ of the form
$Mv \in \mathbb{Z}_2^{|E_1| + |E_2|}$, where $v \in \mathbb{Z}_2^{m + |E_2|}$, 
and $M$ is the $(|E_1|+|E_2|) \times (m + |E_2|)$-matrix 
$( b_1, b_2, \ldots , b_m, d_1, \ldots d_{|E_2|})$ with entries in 
$\mathbb{Z}_2$.

Applying the transformation $R = M^T Q M$ we obtain an $(m+|E_2|)\times(m+|E_2|)$-matrix 
satisfying that {\em (i)} the input vectors $v \in \mathbb{Z}_2^{m+|E_2|}$ are
in one-to-one correspondence with the admissible colourings in $A_{\theta}$ and
{\em (ii)} $v^T R v = 0 \mod 2$ if and only if the admissible colouring encoded 
by $v$ has an even number of tetrahedra of type 
$(1,\frac{1}{2},\frac{1}{2},1,\frac{1}{2},\frac{1}{2})$.

Following the proof of \cite[Lemma~6.29]{Lidl1997Finite}, we can transform $R$ into
an equivalent quadratic form of type direct sum of one of the three 
non-degenerate forms given in Lemma~\ref{lem:quadform}
in $\ell$ indeterminants, and the all-zero quadratic form in 
$m+|E_2|-\ell$ indeterminants.
The number of solutions for the non-degenerate part now follows from 
\cite[Theorem 6.32]{Lidl1997Finite}, 
which is all we need to evaluate
$$ \sum \limits_{\hat{\theta} \in A_{\theta}} |\tri|_{\hat{\theta}}.$$

\subsection{Fixed Parameter Tractable Algorithm in $\beta_1$}
\label{ssec:fpt}

The fixed parameter tractable algorithm to compute $\tv_{4,q}(\tri)$ simply 
consists in running the procedure described in Section~\ref{ssec:poly} for every $1$-cohomology class in $\tri$, 
and sum up all partial sums for all $1$-cohomology classes. A basis for the 
$1$-cohomology group of a triangulation may be computed in polynomial time, 
and the cohomology classes may be enumerated efficiently. 
Moreover, we can sum up the contributions from the trivial cohomology class,
and cohomology classes $\theta$ with even and odd Euler characteristic surfaces 
$S_{\theta}$ separately, resulting in the more powerful invariant 
$(\tv_{4,q} (\tri)_{\nu})_{\nu \in \{0,1,2\}}$ as defined by Matveev
\cite[Section 8.1.5]{matveev03-algms}. These three invariants sum up to $\tv_{4,q}$, but considering them separately yields to a stronger topological invariant than $\tv_{4,q}$.

\paragraph{Correctness of the algorithm.} Following Lemma~\ref{lem:adm3}, every
colouring $\hat{\theta} \in \adm (\tri,4)$ reduces to a unique colouring 
$\theta \in \adm (\tri,3)$ and can thus be associated to a unique $1$-cohomology
class (note that $\tri$ has only one vertex) \cite{Maria15TVAlgos}. By 
Lemma~\ref{lem:octopi} all colourings associated to $\theta$ are assigned the 
same weight up to a sign. By definition of the sets $E_0$, $E_1$, and $E_2$ 
every colouring in $\adm (\tri,4)$ reducing to $\theta$ is considered, and by
the definition of the quadratic form, the number of colourings with a positive
weight equal the number of solutions of the quadratic form. Thus all admissible
colourings are considered with their proper weight.

\paragraph{Running time of the algorithm.} 
Given an $n$-tetrahedron triangulation $\tri$, we transform $\tri$ into
a $1$-vertex $n'$-tetrahedron triangulation $\tri'$, $n' \leq n$ in 
$O(n^3)$ time, using a slight adaptation of the algorithm for knot complements
presented in \cite[Lemma 6]{burton12-unknot}.\footnote{In case $\tri$ contains a two-sided projective 
plane and the algorithm fails, this fact will be detected by the algorithm.}

Computing admissible colourings $\adm (\tri, 3)$ requires solving a linear 
system which can be done in $O(n'^3)$ time. By 
\cite[Proposition 1]{Maria15TVAlgos} we have $|\adm (\tri, 3)| = 2^{\beta_1 
(\tri',\mathbb{Z}_2)}$. For each $\theta \in \adm (\tri, 3)$ we can compute 
$|\tri'|_{\theta}$ and determine $E_0$, $E_1$, and $E_2$ in linear time. 
Computing admissible colourings $A_\theta \subset \adm(\tri', 4)$, again, 
requires solving a linear system which, again, requires $O(n'^3)$ time. Finally,
setting up the quadratic form consists of two matrix multiplications and 
transforming it into canonical form requires $O(m + |E_2|)$ variable relabelings
and sparse basis transforms running in $O((m + |E_2|)^2)$ time each.
Altogether the algorithm thus runs in 
$$O(2^{\beta_1 (\tri',\mathbb{Z}_2)} \cdot n^3)$$ 
time.

Additionally, the algorithm has polynomial memory complexity $O(n^2)$.

\section{$\tv_{4,q}$ is not harder than counting}
\label{sec:counting}

The computational complexity of quantum invariants, and in particular its 
connection with the counting complexity class \#P, is of particular interest
to mathematicians. For instance, it establishes deep connections between the
structure of representations of $3$-manifolds or knots and separation of
complexity classes (see Freedman's seminal work~\cite{freedman09-complexity}
for the Jones polynomial).

Computing 
$\tv_{4,1}$ is known to be \#P-hard, via a reduction from 
\#3SAT~\cite{Burton15TuraevViro,kirby04-nphard}. We prove a converse result here, specifically that computing 
$\tv_{4,q}$, $q \in \{1,3\}$, on an $n$-tetrahedron triangulation $\tri$, 
can be reduced to $\mathrm{poly}(n)$ instances of a counting problem.\footnote{Note that computing $\tv_{4,q}$ is not a \#P problem in nature.}
This is a direct consequence of Lemma~\ref{lem:octopi}. Using the same notations, consider the Laurent polynomial
\[
P_\tri(z) = \sum_{\hat{\theta} \in \adm(\tri,4)} (-1)^{\alpha} 
z^{\chi(S_{\theta})} = \sum_{m \in \Z} a_m z^m .
\]

Note that in this presentation, we
group colourings by the Euler characteristic of their reduction, as opposed to 
Section~\ref{sec:fpt} where they are grouped by reduced colourings. 
Because 
the intersection patterns between a surface $S_{\hat{\theta}}$ and the 
tetrahedra of $\tri$ is constrained to the finite set of cases presented in 
Section~\ref{sec:surfaces}, and $\chi$ is a linear function, the degree of the Laurent polynomial $P_\tri(z)$ is $O(n)$. Naturally, $\tv_{4,1}(\tri) = 
P_\tri(-\sqrt{2})$ and $\tv_{4,3}(\tri) = P_\tri(\sqrt{2})$.

For an integer $m$, let $b^+_m$ (respectively $b^-_m$) be the number of 
colourings $\hat{\theta} \in \adm(\tri,4)$ with an even number of octagons 
(respectively odd) and $\chi(S_\theta) = m$. Consequently, $a_m = b^+_m - 
b^-_m$, and computing the Laurent polynomial $P_\tri(z)$ (and consequently 
computing $\tv_{4,q}(\tri)$) reduces to $\mathrm{poly}(n)$ calls to the 
following problems: 

\vspace{0.2cm}

\begin{minipage}{0.5\textwidth}
{\sc Counting even octagons colourings:}\\
{\bf Input}: $3$-manifold triangulation $\tri$, integer $m$\\
{\bf Output}: $b^+_m$
\end{minipage}
\begin{minipage}{0.5\textwidth}
{\sc Counting odd octagons colourings:}\\
{\bf Input}: $3$-manifold triangulation $\tri$, integer $m$\\
{\bf Output}: $b^-_m$
\end{minipage}

\vspace{0.2cm}

These problems belong to the counting class \#P, as checking if an 
arbitrary assignment of edge colours $\hat{\theta}$ is an admissible 
colouring of $\adm(\tri,4)$ is polynomial time computable, as is
computing the parity of the number of octagons and the Euler characteristic 
of $S_{\theta}$.

\section{Practical significance of the algorithm}
\label{sec:expts}

\paragraph{The power of $\tv_{4,q}$ to distinguish between $3$-manifolds}
The significance of the FPT-algorithm from Section~\ref{ssec:fpt} to compute $\tv_{4,q}$ strongly depends
on the power of $\tv_{4,q}$ to distinguish between non-homeomorphic 
$3$-manifolds.

Since there is no canonical way to quantify this power, we give evidence of the power of $\tv_{4,q}$ along two directions. We first present an infinite family of non-homeomorphic but homotopy equivalent $3$-manifolds, which are provably distinguished by $\tv_{4,1}$. 
We then run practical experiments on large censuses of $3$-manifold triangulations.

The Turaev-Viro invariants of lens spaces have been studied in~\cite{Sokolov97LensSpacesTVInv,Yamada95LensSpacesTVInv}.
\begin{thm}[Based on~\cite{Sokolov97LensSpacesTVInv,Yamada95LensSpacesTVInv}]
  \label{thm:ls}
  Let $L(p,q)$ be the lens space with co-prime parameters $p$ and $q$, 
  $0 < q < p$, and let $k > 0$. Then we have 
  \[
\arraycolsep=2pt
\begin{array}{ccc} 
    \tv_{4,1} ( L(16k,q) ) = \left \{ \begin{array}{ll} 1 & \text{ if } q = \pm 1 \, \mathrm{mod}\, 8 \\
    0 & \text{ otherwise}, 
    \end{array}\right .  & \text{ and }  & 
    \tv_{4,1} ( L(16k-8,q) ) = \left \{ 
    \begin{array}{ll} 1 & \text{ if } q = \pm 3 \, \mathrm{mod}\, 8 \\
    0 & \text{ otherwise}. \end{array}\right .\\
    \end{array}
  \]
\end{thm}

Hence, given the FPT algorithm introduced above, we deduce:
\begin{corollary}
  Given triangulated $3$-manifolds $M$ and $N$ secretly homeomorphic to lens spaces 
  $L(8k,q_1)$ and $L(8\ell,q_2)$, $k,\ell > 0 $, $q_1,q_2 \in \{1,3\}$. 
  Then there exists a polynomial time procedure to decide the 
  homeomorphism problem for $M$ and $N$.
\end{corollary}

\begin{proof}
  We use homology calculations to determine $k$ and $\ell$. If $k \neq \ell$, then
  $M$ and $N$ are not homeomorphic. If $k = \ell$ we know $M$ and $N$ are 
  homotopy equivalent.
  We compute $\tv_{4,1}$ of both $M$  and $N$. Since both $M$ and $N$ have first Betti number equal to $1$, this is 
  a polynomial time procedure. By Theorem~\ref{thm:ls} we conclude that
  $M$ and $N$ are homeomorphic if and only if both values for $\tv_{4,1}$ coincide.
\end{proof}

To determine the power of $\tv_{4,q}$ on a more general level, we run large scale 
experiments on the census of $13\,397$ distinct topological types of minimal 
triangulations of $3$-manifolds with up to $11$ tetrahedra \cite{burton11-genus,matveev03-algms}, and on the 
Hodgson-Weeks census of $11\,031$ topologically distinct hyperbolic manifolds \cite{hodgson94-closedhypcensus}. 
In our experiments we use an implementation of the algorithm presented in 
Section~\ref{ssec:fpt} to compute the finer $3$-tuple of invariants 
$\tv_{4,q}(\tri)_\nu$, $\nu \in \{0,1,2\}$ (see Section~\ref{ssec:fpt}).

The $13\,397$ distinct topologies of the up to $11$ tetrahedra census split into $697$ groups of manifolds with equal integral homology. Combining integral homology with 
$\tv_{4,q,\nu}$, 
$q \in \{1,3\}$, $0 \leq \nu \leq 2$, we are able to split the manifolds further into $1\,205$ groups.
The $11\,031$ manifolds in the Hodgson-Weeks census split into $516$ groups
of equal integral homology. Combining integral homology with $\tv_{4,q,\nu}$
yields $816$ groups of manifolds.

Using $\tv_{4,q}(\tri)_\nu$, $\nu \in \{0,1,2\}$, we are thus able to
distinguish nearly twice as many pairs of $3$-manifolds than with integral 
homology alone.

\paragraph{Indicative timings}
We have for the performance of our algorithm compared to previous state of the
art implementations to compute $\tv_{4,q}$ and integral homology:

\noindent
\begin{tabular}{|l|r|r|r|}
  \hline
  & FPT-alg. from Sec.~\ref{ssec:fpt} & FPT-alg. from \cite{Burton15TuraevViro} & int. homology in Regina~\cite{regina} \\
  \hline
  \hline
  $\leq 11$ tetrahedra census   &$10.96$ sec.&$498$ sec.&$7.72$ sec. \\
  \hline
  Hodgson-Weeks census          &$12.71$ sec.&$1720$ sec.&$14.71$ sec. \\
  \hline
\end{tabular}

\smallskip
In conclusion, the FPT algorithm for $\tv_{4,q}$ presented in this article is of practical importance. Combined with homology, it allows to refine the classification of $3$-manifolds, on our censuses, by a factor of $1.73$ and $1.58$ respectively, at a cost comparable to the computation of homology. We hope this will make $\tv_{4,q}$ a standard pre computation in $3$-manifold topology. We will make the implementation of the algorithm available in Regina~\cite{regina}.

\newpage

\bibliographystyle{plain}
\bibliography{bibliography,pure}

\newpage

\appendix

\section{Homology and cohomology}
\label{app:homology}

In the following section we give a very brief introduction into (co)homology
theory. For more details see \cite{hatcher02-algebraic}.

Let $\tri$ be a generalised $3$-manifold triangulation. For the {\em ring of coefficients} 
$\Z_2 := \Z / 2\Z$, the {\em group of $p$-chains}, $0 \leq p \leq 3$, denoted 
$\Chains_p(\tri,\Z_2)$, of $\tri$ is the group of formal sums of $p$-dimensional faces with 
$\Z_2$ coefficients. The \emph{boundary operator} is a linear operator 
$\partial_p: \Chains_p(\tri,\Z_2) \rightarrow 
\Chains_{p-1}(\tri,\Z_2)$ such that $\partial_p \sigma = \partial_p \{v_0, 
\cdots , v_p\} = \sum_{j=0}^p \{v_0,\cdots ,\widehat{v_j}, \cdots,v_p\}$,
where $\sigma$ is a face of $\tri$, $\{v_0, \ldots, v_p\}$ represents 
$\sigma$ as a face of a tetrahedron of $\tri$ in local vertices $v_0, 
\ldots, v_p$, and $\widehat{v_j}$ means $v_j$ is deleted from the list. Denote 
by $\Cycles_p(\tri,\Z_2)$ and $\Boundaries_{p-1}(\tri,\Z_2)$ the kernel and the 
image of $\partial_p$ respectively. Observing $\partial_p \circ 
\partial_{p+1}=0$, we define the {\em $p$-th homology group} $\Hom_p(\tri,\Z_2)$ of 
$\tri$ by the quotient $\Hom_p(\tri,\Z_2) = \Cycles_p(\tri,\Z_2)/
\Boundaries_p(\tri,\Z_2)$. 

Whenever the ring of coefficients is a field (eg. as above) homology groups are 
vector spaces, otherwise they are modules. If the ring of coefficients is equal to
the integers we refer to them as {\em integral homology groups}. For each finite 
field $\mathbb{F}$, integral homology groups determine homology groups with 
coefficients in $\mathbb{F}$ by virtue of the universal coefficient theorem 
\cite{hatcher02-algebraic}. Hence, as a topological invariant they are at least 
as powerful as homology with coefficients in $\mathbb{F}$.

\medskip
The concept of {\em cohomology} is in many ways dual to homology, but more 
abstract and endowed with more algebraic structure. It is defined in the 
following way: The {\em group of $p$-cochains} $\Chains^p(\tri,\Z_2)$ is the
formal sum of linear maps of $p$-dimensional faces of $\tri$ into $\Z_2$. The 
\emph{coboundary operator} is a linear operator $\delta^p: 
\Chains^{p-1}(\tri,\Z_2) \rightarrow \Chains^{p}(\tri,\Z_2)$ such that 
for all $\phi \in \Chains^{p-1}(\tri,\Z_2)$ we have $\delta^p (\phi) = 
\phi \circ \partial_p$. As above, {\em $p$-cocycles} are the elements in the
kernel of $\delta^{p+1}$, {\em $p$-coboundaries} are elements in the image
of $\delta^{p}$, and the \emph{$p$-th cohomology group} 
$\Hom^p(\tri,\Z_2)$ is defined as the $p$-cocycles factored by the 
$p$-coboundaries.

The exact correspondence between elements of homology and cohomology is best
illustrated by {\em Poincar\'e duality} stating that for closed $d$-manifold 
triangulations $\tri$, $\Hom^p(\tri,\Z_2)$ and $\Hom_{d-p}(\tri,\Z_2)$ are
dual as vector spaces. 

For instance, let $S$ be a $2$-cycle in $\tri$ 
representing a class in $\in \Hom_{2}(\tri,\Z_2)$.
We can perturb $S$ such that it contains no vertex of $\tri$ and intersects
every tetrahedron of $\tri$ in a single triangle (separating one vertex from the
other three) or a single quadrilateral (separating pairs of vertices).
It follows that every edge of $\tri$ intersects $S$ in $0$ or $1$ points.
Then the $1$-cochain defined by mapping every edge intersecting $S$ to $1$ and
mapping all other edges to $0$ represents the Poincar\'e dual of $S$ in 
$ \Hom^1(\tri,\Z_2) $. We will use this exact duality to switch between 
admissible colourings in $\adm(\tri,3)$ and surfaces defined by these colourings
(see Section~\ref{sec:surfaces}).

\section{Weight formulas for Turaev-Viro invariants}
\label{app:weight}

In this section, we introduce the weight formulas for the original Turaev-Viro 
invariants $\tv_{r,q}$ defined in~\cite{turaev92-invariants}. For this, let $r$ 
and $q$ be two integers, such that $r \geq 3 $ and $0 < q < 2r$, with 
$\gcd(r,q) = 1$.

Our notation differs slightly from Turaev and Viro \cite{turaev92-invariants};
most notably, Turaev and Viro do not consider
triangle weights $|f|_\theta$, but instead incorporate an additional
factor of $|f|_\theta^{1/2}$ into each tetrahedron weight
$|t|_\theta$ and $|t'|_\theta$ for the two tetrahedra $t$ and $t'$ containing $f$.
This choice simplifies the notation and avoids unnecessary
(but harmless) ambiguities when taking square roots.

Let $\zeta = e^{i \pi q / r} \in \C$.  Note that our conditions imply that
$\zeta$ is a $(2r)$th root of unity, and that $\zeta^2$ is a
\emph{primitive} $r$th root of unity; that is,
$(\zeta^2)^k \neq 1$ for $k=1,\ldots,r-1$.
For each positive integer $i$, we define
$[i] = (\zeta^i-\zeta^{-i})/(\zeta-\zeta^{-1})$ and,
as a special case, $[0] = 1$.
We next define the ``bracket factorial''
$[i]! = [i]\,[i-1] \ldots [0]$.
Note that $[r] = 0$, and thus $[i]! = 0$ for all $i \geq r$.

We give every vertex constant weight
\begin{equation*}
|v|_\theta = \frac{\left|\zeta-\zeta^{-1}\right|^2}{2r} ,
\end{equation*}
and to each edge $e$ of colour $i \in I$ (i.e.,
for which $\theta(e) = i$) we give the weight
\begin{equation*}
|e|_\theta = (-1)^{2i} \cdot [2i+1].
\end{equation*}
A triangle $f$ whose three edges have colours $i,j,k \in I$ is assigned the
weight
\[ |f|_\theta = (-1)^{i+j+k} \cdot
    \frac{[i+j-k]! \cdot [i+k-j]! \cdot [j+k-i]!}{[i+j+k+1]!}. \]
Note that the parity condition and triangle inequalities ensure that
the argument inside each bracket factorial is a non-negative integer.

\begin{figure}
    \begin{center}
        \includegraphics[width = .15\textwidth]{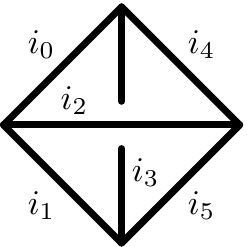}
        \caption{Edge colours of a tetrahedron. \label{fig:tet}}
    \end{center}
\end{figure}

Finally, let $t$ be a tetrahedron with edge colours
$i_0,i_1,i_2,i_3,i_4,i_5$ as indicated in Figure~\ref{fig:tet}. In particular,
the four triangles surrounding $t$ have colours
$(i_0,i_1,i_3)$, $(i_0,i_2,i_4)$, $(i_1,i_2,i_5)$ and $(i_3,i_4,i_5)$,
and the three pairs of opposite edges have colours
$(i_0,i_5)$, $(i_1,i_4)$ and $(i_2,i_3)$.  We define
\begin{align*}
\tau_\phi(t,z) &=
    [z-i_0-i_1-i_3]! \cdot [z-i_0-i_2-i_4]! \cdot
    [z-i_1-i_2-i_5]! \cdot [z-i_3-i_4-i_5]!\,, \\
\kappa_\phi(t,z) &= [i_0+i_1+i_4+i_5-z]! \cdot
                  [i_0+i_2+i_3+i_5-z]! \cdot
                  [i_1+i_2+i_3+i_4-z]!
\end{align*}
for all integers $z$ such that the bracket factorials above all have non-negative
arguments; equivalently, for all integers $z$ in the range
$\zmin \leq z \leq \zmax$ with
\begin{align*}
\zmin &= \max\{i_0+i_1+i_3,\ i_0+i_2+i_4,\ i_1+i_2+i_5,\ i_3+i_4+i_5\}\,; \\
\zmax &= \min\{i_0+i_1+i_4+i_5,\ i_0+i_2+i_3+i_5,\ i_1+i_2+i_3+i_4\}.
\end{align*}
Note that, as before, the parity condition ensures
that the argument inside each bracket factorial above is an integer.
We then declare the weight of tetrahedron $t$ to be
\begin{equation*}
|t|_\phi = \sum_{\zmin \leq z \leq \zmax}
    \frac{(-1)^z \cdot [z+1]!}{\tau_\phi(t,z) \cdot \kappa_\phi(t,z)},
\end{equation*}

\bigskip
Note that all weights are polynomials on $\zeta$ with rational coefficients, 
where $\zeta = e^{i \pi q/r}$.
\end{document}